\documentclass[11pt,english]{article}

\usepackage[margin = 2.35cm]{geometry}

\usepackage{amsthm}
\usepackage{amsmath}
\usepackage{amssymb}
\usepackage{setspace}
\usepackage{mathtools}
\usepackage{graphicx}
\usepackage[hidelinks]{hyperref}
\usepackage{thm-restate}
\usepackage{cleveref}
\usepackage{hyperref}
\usepackage{enumitem}
\usepackage{framed}
\usepackage{subcaption}

\usepackage[square,sort,comma,numbers]{natbib}
\usepackage{floatrow}

\theoremstyle{plain}

\newtheorem*{thm*}{Theorem}
\newtheorem{thm}{Theorem}
\Crefname{thm}{Theorem}{Theorems}

\newtheorem*{lem*}{Lemma}
\newtheorem{lem}[thm]{Lemma}
\Crefname{lem}{Lemma}{Lemmas}

\newtheorem*{claim*}{Claim}
\newtheorem{claim}[thm]{Claim}
\crefname{claim}{Claim}{Claims}
\Crefname{claim}{Claim}{Claims}

\newtheorem{prop}[thm]{Proposition}
\Crefname{prop}{Proposition}{Propositions}

\newtheorem{cor}[thm]{Corollary}
\crefname{cor}{Corollary}{Corollaries}

\crefname{conj}{Conjecture}{Conjectures}

\Crefname{qn}{Question}{Questions}

\newtheorem{obs}[thm]{Observation}
\Crefname{obs}{Observation}{Observations}

\Crefname{ex}{Example}{Examples}

\theoremstyle{definition}

\Crefname{prob}{Problem}{Problems}

\Crefname{defn}{Definition}{Definitions}

\theoremstyle{remark}

\renewenvironment{proof}[1][]{\begin{trivlist}
\item[\hspace{\labelsep}{\bf\noindent Proof#1.\/}] }{\qed\end{trivlist}}

\newcommand{\remove}[1]{}

\newcommand{\ceil}[1]{
    \left \lceil #1 \right \rceil
}
\newcommand{\floor}[1]{
    \left \lfloor #1 \right \rfloor
}

\newcommand{\eps}{\varepsilon}
\renewcommand{\P}{\mathcal{P}}

\DeclareMathOperator{\bip}{bip}

\newcommand{\CM}{\mathcal{CM}}
\newcommand{\F}{\mathcal{F}}
\newcommand{\rbip}[2]{r^{\bip}(#2, #1)}
\newcommand{\M}{\mathcal{M}}
\newcommand{\G}{\mathcal{G}}

\begin{document}


\title{\vspace{-0.9cm} Multicolour bipartite Ramsey number of paths}

\author{
	Matija Buci\'c\thanks{
	    Department of Mathematics, 
	    ETH, 
	    8092 Zurich;
	    e-mail: \texttt{matija.bucic}@\texttt{math.ethz.ch}.
	}
	\and
    Shoham Letzter\thanks{
        ETH Institute for Theoretical Studies,
        ETH,
        8092 Zurich;
        e-mail: \texttt{shoham.letzter}@\texttt{eth-its.ethz.ch}.
        Research supported by Dr.~Max
        R\"ossler, the Walter Haefner Foundation and by the ETH Zurich Foundation.
    }
    \and
	Benny Sudakov\thanks{
	    Department of Mathematics, 
	    ETH, 
	    8092 Zurich;
	    e-mail: \texttt{benjamin.sudakov}@\texttt{math.ethz.ch}.     
		Research supported in part by SNSF grant 200021-175573.
	}
}
\date{}

\maketitle

\begin{abstract}
    \setlength{\parskip}{\smallskipamount}
    \setlength{\parindent}{0pt}
    \noindent
    
    \vspace{-0.9cm} 
    The \textit{$k$-colour bipartite Ramsey number} of a bipartite graph $H$ is the least integer $N$ for which every $k$-edge-coloured complete bipartite graph $K_{N,N}$ contains a monochromatic copy of $H$. 
    The study of bipartite Ramsey numbers was initiated over 40 years ago by Faudree and Schelp and, independently, by Gy\'arf\'as and Lehel, who determined the $2$-colour bipartite Ramsey number of paths. Recently the $3$-colour Ramsey number of paths and (even) cycles, was essentially determined as well. Improving the results of DeBiasio, Gy\'arf\'as, Krueger, Ruszink\'o, and S\'ark\"ozy, in this paper we determine asymptotically the $4$-colour bipartite Ramsey number of paths and cycles. We also provide new upper bounds on the $k$-colour bipartite Ramsey numbers of paths and cycles which are close to being tight.
\end{abstract}

\section{Introduction}
	Ramsey theory refers to a large body of mathematical results, which roughly say that any sufficiently large structure is guaranteed to have a large well-organised substructure. For example, the celebrated theorem of Ramsey \cite{ramsey1929problem} says that for any fixed graph $H$, every $k$-edge-colouring of a sufficiently large complete graph contains a monochromatic copy of $H$. The \emph{$k$-colour Ramsey number of $H$} is defined to be the smallest order of a complete graph satisfying this property.
	
	Despite significant attention paid to Ramsey problems, there are very few examples of families of graphs whose Ramsey numbers are known exactly, or even just asymptotically. An early example of an exact Ramsey result was obtained in 1967 by Gerencs\'er and Gy\'arf\'as \cite{ramsey2path}, who determined the $2$-colour Ramsey number of paths. Ramsey numbers of paths and cycles have since been studied extensively, and they are known precisely for two and three colours (in most cases only for sufficiently large $n$), see \cite{faudree-schelp-cycles,rosta-cycles-73,luczak99-con-match,kohayakawa-simonovits-skokan,ramsey3path,benevides-skokan}. However, despite extensive research, less is known for more than three colours. A rare exception is a recent result of Jenssen and Skokan \cite{jenssen-skokan}, who showed that the $k$-colour Ramsey number of an \emph{odd} cycle $C_n$ is exactly $2^{k-1}(n-1)+1$ for all sufficiently large $n$; interestingly, this does not hold for all $k$ and $n$, see Day and Johnson \cite{day-johnson}. For a path $P_n$, the $k$-colour Ramsey number is known to be at least $(k - 1 + o(1))n$ (see Yongqi, Yuansheng, Fengand and Bingxi \cite{Yongqi}), and at most $(k - 1/2 + o(1))n$ (see Knierim and Su \cite{knierim-su}); the same bounds also hold for \emph{even} cycles $C_n$.

	Over the years, many generalisations of Ramsey numbers have been considered (an excellent survey \cite{conlon2015recent} by Conlon, Fox and Sudakov contains many examples); one natural example that we consider here is obtained by replacing the underlying complete graph by a complete bipartite graph. In particular, the \textit{$k$-colour bipartite Ramsey number} of a bipartite graph $H$ is the least integer $N$ such that in any $k$-colouring of the complete bipartite graph $K_{N,N}$ there is a monochromatic copy of $H$. 
	
	The study of bipartite Ramsey numbers was initiated in the early 70s by Faudree and Schelp \cite{faudree-schelp} and independently Gy\'arf\'as and Lehel \cite{gyarfas2path} who determined the $2$-colour bipartite Ramsey numbers of paths; see also \cite{zhang1,zhang2,joubert} for some results regarding the natural extension to cycles. Recently, we \cite{3-bip-path} determined, asymptotically, the $3$-colour Ramsey number of a path or a cycle of length $2n$, showing that it is equal to $(3 + o(1))n$. Further related results were obtained by \cite{luo-peng,liu-peng}.
	
	Similarly to the standard Ramsey numbers, less is known regarding bipartite path-or-cycle Ramsey numbers for more than three colours. The best known lower bound for the $k$-colour bipartite Ramsey number of a path or a cycle of length $2n$ is $5n$ for $k = 4$ and $(2k-4)n$ for $k \ge 5$, while the best known upper bound is $k(1 + \sqrt{1 - 2/k} + o(1))n$ (which is roughly $(2k - 1 + o(1))n$ for large $k$). Both results are due to DeBiasio, Gy\'arf\'as, Krueger, Ruszink\'o and S\'ark\"ozy \cite{debiasio-multicolour} who also say that obtaining improvement to either of these bounds would be very interesting.
	
	In this paper we achieve this, improving the best known upper bound for all $k \ge 4$.
	\begin{thm} \label{thm:path-2k-3}
	    Let $k \ge 4$. The $k$-colour bipartite Ramsey number of a cycle or path of order $2n$ is at most $(2k - 3 + o(1))n$.
    \end{thm}
    As an immediate corollary (using the lower bound mentioned above, see also \Cref{obs:lower-bds}), we determine, asymptotically, the $4$-colour bipartite Ramsey number of a path or a cycle.
  	\begin{cor} \label{cor:4-col}
	    The $4$-colour bipartite Ramsey number of a cycle or path of order $2n$ is $(5 + o(1))n$.
    \end{cor}  
     
    For larger $k$, we are able to use our methods to improve on \Cref{thm:path-2k-3}.
	\begin{thm} \label{thm:path-2k-3.5}
	    Let $k \ge 8$. The $k$-colour bipartite Ramsey number of a cycle or path of order $2n$ is at most $\left(2k - 3.5 + \frac{1}{k-2} + o(1)\right)n$.
    \end{thm}
    
    We believe that the correct bound may be close to the latter theorem, namely about $(2k - 3.5 + o(1))n$. As evidence, we give such a lower bound for $k = 5$ which was initially found with the help of a computer. Despite our best efforts, we have not been able (either by computer search or by hand) to extend this bound to larger values of $k$. Nevertheless, we believe that such a bound may hold for all $k \ge 5$.
    \begin{thm} \label{thm:lower-5}
        The $5$-colour bipartite Ramsey number of a cycle or path of order $2(n+1)$ is larger than $6.5n$.
    \end{thm}
	
    \subsection{Organisation of the paper}	

        In our proofs we use \L{}uczak's method of converting problems about cycles and paths to problems about connected matchings. The method requires us to work with the so-called \emph{reduced graph}, obtained by applying Szemer\'edi's regularity lemma, and look for a \textit{monochromatic connected matching} (i.e.\ a matching that is contained in a monochromatic component) in this graph, which is \emph{almost} complete bipartite. In our previous paper \cite{3-bip-path}, we showed that it suffices to consider connected matchings in complete bipartite graphs; we give more details on these two topics in \Cref{sec:regularity}. Our main efforts are thus devoted to the problem of finding monochromatic connected matchings in $k$-coloured complete bipartite graphs, which we consider in \Cref{sec:cm}. We describe the lower bound mentioned in \Cref{thm:lower-5}, as well as the best known lower bounds for general $k$, in \Cref{sec:lower-bounds}. We conclude the paper in \Cref{sec:conc-remarks} with some remarks and open problems.

\section{Monochromatic connected matchings in \texorpdfstring{$K_{n,n}$}{Kn,n}}\label{sec:cm}
    
    A \emph{connected matching} in a graph $H$ is a matching that is contained in a connected component of $H$, and a \emph{connected $k$-matching} is a connected matching that consists of $k$ edges. In an edge-coloured graph $H$, a \emph{$c$-coloured connected matching} is a connected matching in the subgraph of $H$ whose edges are the edges of $H$ of colour $c$. We shall use the notation $\CM(n)$ to denote the family of connected $n$-matchings. We write $G \overset{k}{\longrightarrow} \F$ if in every $k$-colouring of $G$ there is a monochromatic copy of some graph in $\F$, where $\F$ is some family of graphs. In particular, $G \overset{k}{\longrightarrow} \CM(n)$ means that in every $k$-colouring of $G$ there is a monochromatic connected $n$-matching. Let $r_k(n)$ denote the smallest integer $N$ such that for any $k$-colouring of $K_{N,N}$ there is a monochromatic connected $n$-matching. In this section we shall prove upper bounds on $r_k(n+1)$, from which \Cref{thm:path-2k-3,thm:path-2k-3.5} can be deduced, using {\L}uczak's method and our \Cref{thm:cm-almost-complete}. We shall make use of the following result which we proved in \cite{3-bip-path}.
    \begin{thm} \label{thm:cm-3}
        $r_3(n+1) = 3n+1$.
    \end{thm}
    
    We begin by setting the stage and proving some simple propositions. Throughout this section, let $G$ be a $k$-coloured $K_{N,M}$, and denote its bipartition by $\{X, Y\}$, where $|X| = N$ and $|Y| = M$. 
    We assume that $G$ does not have a monochromatic connected $(n+1)$-matching. Let $C_1,\ldots, C_t$ denote all the monochromatic components of $G$ and let $c(i)$ denote the colour of component $C_i$. By K\"onig's theorem, each monochromatic component has a cover of size at most $n$; fix a minimum cover $\tau_i$ for $C_i$, so $|\tau_i| \le n$. Let us start by observing a simple property.
    
    \begin{claim}\label{claim:high-deg-col}
        Let $v$ be a vertex and suppose that it is incident with at least $n+1$ edges of colour $c$. Then $v \in \tau_i$ for some $i$ with $c(i) = c$.
    \end{claim}
    \begin{proof}
        We call the colour $c$ red. Let $C_i$ be the red component that contains $v$ (so $v \in V(C_i)$ and $c(i) = c$). 
        We know, since $|\tau_i|\le n,$ that at least one of the red neighbours of $v$ is not in $\tau_i,$ say $w$. Since the edge $vw$ is covered by $\tau_i$, it follows that $v \in V(C_i)$, as claimed.
    \end{proof}
    
    We will always have $N,M \ge kn+1$ in our arguments since otherwise it is easy to find a $k$-colouring of $K_{N, M}$ without monochromatic connected $(n+1)$-matchings. This implies that every vertex has degree at least $n+1$ in some colour. By \Cref{claim:high-deg-col}, every vertex belongs to some $\tau_i$.
    We say that a vertex is \textit{special} if it belongs to exactly one $\tau_i$, in which case we say that it is \emph{special of colour $c(i)$}. Let us prove some simple properties of special vertices. 
    We call a component \textit{special} if it contains a special vertex of its colour. 
    
    \begin{prop}\label{prop:basic-special}
        The following assertions hold.
        \begin{enumerate}[label=\rm (\arabic{*}), ref=(\arabic{*})]
            \item \label{itm:special-deg} 
                Let $v \in X$ be a red special vertex. Then $v$ has red degree at least $M-(k-1)n.$ Similarly, if $v \in Y$ is a red special vertex then it has red degree at least $N - (k-1)n$.
            \item \label{itm:special-sides} 
                If there are two distinct special red components then all red special vertices belong to the same side of the bipartition.
        \end{enumerate}
    \end{prop}
    \begin{proof}
        For \ref{itm:special-deg} notice that since $v$ is special in red it has degree at most $n$ in every colour other than red, by \Cref{claim:high-deg-col}. In particular, since $v$ has degree $M$ in $G$ and there are $k-1$ colours other than red, its red degree is at least $M-(k-1)n.$
    
        For \ref{itm:special-sides} assume to the contrary that there are two red special vertices $v$ and $w$, where $v \in X$ and $w \in Y$ and they belong to different red components. The edge $vw$ is not coloured red since $v$ and $w$ belong to different red components, so say that $vw$ is blue. But if $C_i$ is the blue component containing the edge $vw$ then at least one of $v$ and $w$ belongs to $\tau_i$, and in particular it is not a red special vertex, a contradiction. 
    \end{proof}
    
    We say that a vertex $v\in \tau_i$ is \textit{somewhat special} in colour $c(i)$ if it is not $c(i)$ special and there is a $c(i)$ special vertex in $C_i$ on the other side of the bipartition from $v.$ For a component $C_i$ let $t_i$ denote the number of somewhat special vertices of colour $c(i)$ in $C_i$, and let $T = \sum_{i=1}^t t_i.$ The double counting argument in the next lemma gives us a lower bound on the number of special vertices and is at the heart of most of our arguments.

	Before proceeding, let us give some intuition behind the concept of somewhat special vertices. Given a red component $C_i$, which contains a red special vertex in part $X$ of the bipartition, by \Cref{prop:basic-special}~\ref{itm:special-deg} we know that $Y \cap C_i$ is pretty big. Notice that all the vertices of $\tau_i \cap Y$ are somewhat special and moreover $C_i \setminus \tau_i$ spans no red edges, thus if we can show that $C_i$ contains only a few somewhat special vertices we obtain a rather large set not spanning any red edges. This set can be used to induct on the number of colours. The term depending on the somewhat special vertices in the following lemma allows us to show that there will indeed be few somewhat special vertices in each special component. 
    
	\begin{lem}\label{lem:special-vertices}
	    There are at least $NM/n-(N+M)(k-2)+T(\min(N,M)/n-k)$ special vertices.
	\end{lem}
	
	\begin{proof}
	    Let $\P$ be the set of pairs $(v,e)$ where $v\in \tau_i$ and $e$ is an edge of colour $c(i)$ incident to $v$, for some $i$.
	    
	    Given an edge $e$, let $c$ be its colour, and let $C_i$ be the component of colour $c$ that contains $e$. Since $\tau_i$ is a cover of $C_i$, it contains one of the vertices in $e$, so there is at least one pair in $\P$ that contains $e$. In particular, 
	    \begin{equation} \label{eqn:P-lower-bound}
	        |\P| \ge NM.
	    \end{equation}
    
        On the other hand, given $i$, write $x_i = |\tau_i \cap X|$ and $y_i = |\tau_i \cap Y|$, and let $z_i = |(V(C_i) \cap X) \setminus \tau_i|$ and $w_i = |(V(C_i) \cap Y) \setminus \tau_i|$. In particular, $x_i+y_i = |\tau_i|\le n$ and $x_i+y_i+z_i+w_i=|C_i|$. Moreover, as $\tau_i$ is a minimum cover of $C_i$ of size $x_i + y_i$, and $V(C_i) \cap X$ is another cover of $C_i$, of size $x_i + z_i$, we find that $z_i \ge y_i$; similarly, $w_i \ge x_i$. 
        
		In fact, we know more: if there is a special vertex of $C_i$ in $Y$ then, by \Cref{prop:basic-special}~\ref{itm:special-deg}, $x_i+z_i \ge N-(k-1)n$ implying that $z_i-y_i \ge N-(k-1)n-x_i-y_i \ge N-kn.$ Similarly if $C_i$ has a special vertex in $X$ then $w_i-x_i \ge M-kn.$ We claim that the following inequality holds, where $t_i$ is the number of somewhat special vertices in $C_i$.
        \begin{align}\label{equ:special-on-the-other-side}
            x_i(z_i-y_i)+y_i(w_i-x_i) & \ge t_i(\min(N,M)-kn).
        \end{align} 
        Indeed, if $C_i$ is not a special component, then $t_i = 0$ and the inequality follows from the observation that $z_i \ge y_i$ and $w_i \ge x_i$. Now suppose that $C_i$ is a special component, but all its special vertices (in colour $c(i)$) appear in one side of the bipartition, say $X$. Then $t_i \le y_i$, hence, as explained above, $y_i(w_i - x_i) \ge y_i(M - kn) \ge t_i(\min(N, M) - kn)$. Finally, if $C_i$ contains special vertices in both sides of the bipartition, then $x_i(z_i-y_i)+y_i(w_i-x_i) \ge x_i (N - kn) + y_i (M - kn) \ge (x_i + y_i) (\min(N, M) - kn) \ge t_i (\min(N, M) - kn)$, where the last inequality holds because, trivially, $t_i \le x_i + y_i$.
        
        Now let us count the number of pairs $(v, e)$ in $\P$ for which $e$ is an edge of $C_i$. Let $e = vw$ be an edge of $C_i$. If both $v$ and $w$ are in $\tau_i$ then $e$ appears in two pairs of $\P$, and otherwise it appears in exactly one pair in $\P$. As there are at most $x_i y_i$ edges of the former type and at most $x_i w_i + y_i z_i$ of the latter type, we find that the number of pairs in $\P$ that contain an edge from $C_i$ is bounded from above by
        \begin{align*}
            2x_iy_i+z_iy_i+w_ix_i  & =  (z_i+w_i)(x_i+y_i)-x_i(z_i-y_i)-y_i(w_i-x_i)  \\
            & \le (|C_i|-x_i-y_i)n-t_i(\min(N,M)-kn) \\
            & = (|C_i|-|\tau_i|)n-t_i(\min(N,M)-kn).
        \end{align*}
        where we used the inequality in \eqref{equ:special-on-the-other-side} for the inequality. Summing over all components we obtain:
        \begin{align} \label{eqn:P-upper-bound}
        \begin{split}
            |\P| 
            &\le \sum_{i=1}^t \left(|C_i|-|\tau_i|-t_i\left(\frac{\min(N,M)}{n}-k\right)\right)n \\
            &= (N+M)kn-\sum_{i=1}^t |\tau_i|n-T\left(\frac{\min(N,M)}{n}-k\right)n,
        \end{split}
        \end{align}
        where we used the fact that the sum of sizes of components of any given colour is $N+M$, because every vertex belongs to exactly one component of each colour. 
        
        By comparing \eqref{eqn:P-lower-bound} and \eqref{eqn:P-upper-bound} we find that
        \begin{equation} \label{eqn:tau-upper-bound}
            \sum_{i=1}^t |\tau_i| \le (N+M)k - \frac{NM}{n}-T\left(\frac{\min(N,M)}{n}-k\right).
        \end{equation}
        Finally, denote the number of special vertices by $s$. Since every vertex is in at least one set $\tau_i$ and every non-special vertex is in at least two such sets, we have $\sum_{i=1}^t |\tau_i| \ge 2(N+M)-s$. Combining this with \eqref{eqn:tau-upper-bound} we obtain the inequality
        \begin{align*}
            s \ge \frac{NM}{n}-(N+M)(k-2)+T\left(\frac{\min(N,M)}{n}-k\right),
        \end{align*}
        as desired.
	\end{proof}

	Let us now use \Cref{lem:special-vertices} to obtain an upper bound on $r_k(n+1)$ which is tight for $k=4$ and is not far from being tight in general.
	
	\begin{thm} \label{thm:2k-3}
		For $k \ge 3$ we have $r_k(n+1) \le (2k-3)n+1.$
	\end{thm}
    
    \begin{proof}
        Our proof is by induction on $k$. The base case of $k=3$ follows from our previous result \Cref{thm:cm-3}. 
        For the induction step, let $k \ge 4$ and $N = (2k-3)n + 1$. Assume that $r_{k-1}(n+1)\le (2k-5)n+1 = N - 2n$, so every $(k-1)$-coloured $K_{N-2n,N-2n}$ contains a monochromatic connected $(n+1)$-matching.
    
        Suppose towards a contradiction that $G$ is a $k$-coloured $K_{N,N}$ that does not have an connected $(n+1)$-matching. We shall be using the notation introduced earlier in the section. In particular, we assume that $C_1, \ldots, C_t$ is the list of monochromatic components of $G$ and $\tau_i$ is some minimum cover of $C_i$, for every $1 \le i \le t$.
        By \Cref{lem:special-vertices}, there are at least $N^2/n-2N(k-2) > N( 2k-3 - 2(k-2)) = N \ge kn+1$ special vertices. In particular, there are at least $n+1$ special vertices of the same colour, say red. 
     
        \begin{claim}
            Precisely two red components contain red special vertices. 
        \end{claim}
        \begin{proof}
            Notice that any red special vertex is a cover vertex of its own red component. Since each component contains at most $n$ cover vertices, at least two distinct red components contain a red special vertex. Together with \Cref{prop:basic-special}~\ref{itm:special-sides}, it follows that all the red special vertices are in the same side of the bipartition of $G$, say $X$. By \ref{itm:special-deg} in the same proposition, every red special vertex has red degree at least $(2k-3)n+1-(k-1)n=(k-2)n+1$. In particular, every red special component contains at least $(k-2)n+1$ vertices of $Y$ and these sets are disjoint for distinct components. But $|Y| = (2k-3)n+1 < 3((k-2)n+1)$, so at most two red components contain red special vertices.
        \end{proof}
     
        Let $C_i$ and $C_j$ be the red components that contain red special vertices. Recall that all red special vertices are in the same side of $G$ (by \Cref{prop:basic-special}~\ref{itm:special-sides}), say $X$. Let $X':=X \setminus (\tau_i \cup \tau_j)$ and $Y':=(Y \cap (V(C_i) \cup V(C_j))) \setminus (\tau_i \cup \tau_j).$ Note that $|X'|\ge |X|-2n =N-2n$. Moreover, since all the red special vertices are contained in $(\tau_i \cup \tau_j) \cap X$, and there are at least $n+1$ such vertices, we have $|(\tau_i \cup \tau_j) \cap Y| \le n-1$. Since each of $C_i$ and $C_j$ contains at least $(k-2)n+1$ vertices of $Y$ (by \Cref{prop:basic-special}~\ref{itm:special-deg}), we have $|Y'| \ge 2((k-2)n+1)-|(\tau_i \cup \tau_j) \cap Y| \ge (2k-5)n+3 > N-2n$. Notice also that since neither $X'$ nor $Y'$ contain a vertex of $\tau_i$ or $\tau_j$, and $Y' \subseteq V(C_i) \cup V(C_j)$, there are no red edges between $X'$ and $Y'.$ This means that $G[X', Y']$ is a $(k-1)$-coloured complete bipartite graph where each side has size at least $N-2n$. By our inductive assumption, $G[X', Y']$ contains a monochromatic connected $(n+1)$-matching, a contradiction.
    \end{proof}
    
    \begin{thm}\label{thm:3.5-2}
        Let $k \ge 4$, $N = \ceil{(2k-3.5)n + 1}$ and $M = (2k-2)n + 1$. Then $K_{N, M} \overset{k}{\longrightarrow} \CM(n+1)$.
    \end{thm}
    \begin{proof}
        Suppose, towards a contradiction, that $G$ is a $k$-coloured $K_{N,M}$ that does not have a monochromatic connected $(n+1)$-matching. As usual, we shall be using the notation introduced earlier in the section. 
        By \Cref{lem:special-vertices} we know that $G$ contains at least the following number of special vertices.
        \begin{align}\label{eqn:specialt2}
        \begin{split}
            & \frac{NM}{n}-(N+M)(k-2) \\
            \ge\,\, & \frac{1}{n} ((2k-3.5)n + 1)((2k-2)n + 1) - ((4k-5.5)n + 2.5)(k-2) \\ 
            =\,\, & 2.5 kn - 4n + 1.5 k - 0.5 + \frac{1}{n}  \\
            >\,\, & (2k-1)n+1,
        \end{split}
        \end{align}
		where in the first inequality we used $N = \ceil{(2k-3.5)n + 1} \le (2k-3.5)n+1.5$ and in the last inequality $k \ge 6$. Note that any special vertex of colour $c$ in $X$ has $c$-degree at least $M - (k-1)n > M/2.$ This means that $X$ does not contain two special vertices of the same colour which belong to distinct components of this colour. Similarly, since every special vertex of colour $c$ in $Y$ has $c$-degree at least $N - (k-1)n > N/3$, $Y$ does not contain three special vertices of colour $c$, each from a distinct component. In particular, by \Cref{prop:basic-special}~\ref{itm:special-sides}, there are at most $2n$ special vertices of any colour and if there are more than $n$ they all lie in $Y.$ Together with \eqref{eqn:specialt2} this implies that every colour must have more than $n$ special vertices, so they all lie in $Y,$ which is a contradiction since $|Y|=(2k-2)n+1<(2k-1)n+1.$  

        For $k=4,5$ we have $5kn/2-4n+2>(2k-2)n+2$ so we know there can be at most one colour which has at most $n$ special vertices and all other colours have all their special vertices in $Y.$ If all colours have more than $n$ then we can reach a contradiction as before, so we assume that exactly $k-1$ colours have more than $n$ special vertices. For any of these $k-1$ colours, if $C_i,C_j$ are the special components of this colour we know $|(C_i \cup C_j)\cap X| \ge (2(k-1)-3)n+1$ so if $|Y \setminus (C_i \cup C_j)| \ge (2(k-1)-3)n+1$ we are done by \Cref{thm:2k-3}. In particular, we may assume that $|Y \cap (C_i \cup C_j)|\ge 3n+1.$ Inheriting notation from the proof of \Cref{lem:special-vertices} this means that $y_i+w_i+y_j+w_j \ge 3n+1$, so (using $w_i \ge x_i$, $x_i + y_i \le n$, and the analogous inequalities for $j$)
		\begin{align*}
			y_i(w_i-x_i)+y_j(w_j-x_j) 
			& \ge y_i\max(w_i+y_i-n,0)+y_j \max(w_j+y_j-n,0) \\
			& \ge \min(y_i, y_j) (y_i + w_i + y_j + w_j - 2n) \\
			& \ge \min(y_i, y_j) (n+1) \\
			& > (y_i + y_j - n)n.
		\end{align*}
		Now, we find that the sum of $x_i(z_i - y_i) + y_i(w_i - x_i)$ over all monochromatic components $C_i$, is at least the sum of $y_i(w_i - x_i)$ over all components with colours with more than $n$ special vertices (using $w_i \ge x_i$ and $z_i \ge y_i$), which by the above inequality is at least $(s - n)n - (k-1)n^2 \ge (k-2)n^2 \ge 2n^2$, where $s$ is the number of special vertices. Continuing as in the proof of \Cref{lem:special-vertices}, we find that the number of special vertices is, in fact, at least $NM/n - (N+M)(k-2) + 2n \ge 2kn + 2$, a contradiction. 
    \end{proof}
    
    In the above two theorems we used a weaker estimate than the one provided by \Cref{lem:special-vertices}, namely we ignored the term that includes $T$, the number of somewhat special vertices.  In the next lemma, which will lead to the proof of our main results, \Cref{thm:path-2k-3,thm:path-2k-3.5}, we shall make use of the stronger statement given by \Cref{lem:special-vertices}. 
    
    \begin{lem}\label{lem:recursion}
        Let $k\ge 5$ and let $N, M$ be integers that satisfy the following conditions.
        \begin{enumerate}
            \item \label{itm:N-M-lower}
                $N, M \ge (2k - 3.5)n$,
            \item \label{itm:NM-lower}
                $NM/n - (k-2)(N+M) > kn$,
            \item \label{itm:recursion-a}
	            $K_{N', M'} \overset{k-1}{\longrightarrow} \CM(n+1)$,
                when $N' = N - 2n$ and $M' = 2M - 2n(k-1) - \frac{2n}{2k-5}$,
            \item \label{itm:recursion-b}
	            $K_{N', M'} \overset{k-1}{\longrightarrow} \CM(n+1)$,
                when $N' = 2N - 2n(k-1) - \frac{2n}{2k-5}$ and $M' = M - 2n$.
        \end{enumerate}
        Then $K_{N, M} \overset{k-1}{\longrightarrow} \CM(n+1)$.
    \end{lem}
      
    \begin{proof}
        Suppose, towards a contradiction, that $G$ is a $k$-coloured $K_{N,M}$ that does not have an connected $(n+1)$-matching. 
		Recall that, by \Cref{prop:basic-special}~\ref{itm:special-deg}, every special vertex of colour $c$ in $X$ has $c$-degree at least $M - (k-1)n > M/3$, thus there are at most two special vertices of colour $c$ in $X$ that belong to distinct components. It follows from this and the corresponding statement for $Y$, as well as \Cref{prop:basic-special}~\ref{itm:special-sides}, that at most two components of any given colour contain special vertices of that colour.
        
        \begin{claim}
            There is a colour in which there are precisely two components that contain special vertices, and which has most $2n/(2k-5)$ somewhat special vertices.
        \end{claim}
        
        \begin{proof}
            By \Cref{lem:special-vertices}, the number of special vertices is at least 
            \begin{equation} \label{eqn:num-special-vs}
                \frac{NM}{n} - (k-2)(N+M) + T\left(\frac{\min(N,M)}{n} - k\right) > kn + T(k - 3.5),
            \end{equation}
            where we used Conditions \ref{itm:N-M-lower} and \ref{itm:NM-lower}.
            Let $S$ be the set of colours in which there are at least $n+1$ special vertices, and denote $m = |S|$. As the number of special vertices is larger than $kn$, we have $m \ge 1$. Call the colour in $S$, with the least number of somewhat special vertices, red, and let $t$ be this number. We shall show that $t \le 2n / (2k - 5)$, an assertion which would complete the proof of the claim.
            
            As $T$ is the number of somewhat special vertices, we have $T \ge mt$. Since there are at most $(k-m)n$ special vertices in colours outside of $S$, the number of special vertices in the colours in $S$ is at least $kn + T(k - 3.5) - (k-m)n \ge mn + mt(k - 3.5)$. It follows that for some colour in $S$, say blue, there are at least $n + t(k - 3.5)$ special vertices. Since the blue special and somewhat special vertices all belong to a minimum cover of one of two blue components, there are at most $2n$ such vertices. It follows that there are at most $n - t(k - 3.5)$ blue somewhat special vertices. By definition of $t$, we find that $n - t(k - 3.5) \ge t$, i.e.\ $t \le 2n/(k-5)$, as needed.
        \end{proof}
        
        Let red be the colour given by the above claim and let $C_i$ and $C_j$ be the red special components. Recall that all red special vertices need to be on the same side of $G$ (by \Cref{prop:basic-special}~\ref{itm:special-sides}). 
        
		Let us first assume that this side is $X$. In this case, $(\tau_i \cup \tau_j) \cap Y$ constitutes the set of all red somewhat special vertices of this colour and $|(\tau_i \cup \tau_j) \cap Y| \le 2n/(2k-5)$ by the above claim. Let $X':=X \setminus (\tau_i \cup \tau_j)$ and $Y':=(Y \cap (C_i \cup C_j)) \setminus (\tau_i \cup \tau_j).$ Note that $|X'|\ge |X|-2n =N-2n$. Since each of $C_i$ and $C_j$ contains at least $M-(k-1)n$ vertices of $Y$ (by \Cref{prop:basic-special}~\ref{itm:special-deg}), we have $|Y'| \ge 2(M-(k-1)n)-|(\tau_i \cup \tau_j) \cap Y| \ge 2(M-(k-1)n)-2n/(2k-5)$. Notice also that since neither $X'$ nor $Y'$ contain a vertex of $\tau_i$ or $\tau_j$, and $Y' \subseteq V(C_i) \cup V(C_j)$, there are no red edges between $X'$ and $Y'.$ This means that $G[X', Y']$ is a $(k-1)$-coloured complete bipartite graph with sides of sizes at least $N-2n$ and $2M-2n(k-1+1/(2k-5)),$ respectively, so by Condition \ref{itm:recursion-a} contains a monochromatic connected $(n+1)$-matching, a contradiction to our assumption that $G$ does not have a monochromatic connected $(n+1)$-matching.
        
		In the other case we analogously obtain $X',Y'$ with $|Y'|\ge M-2n$ and $|X'|\ge 2N-2n(k-1+1/(2k-5))n,$ which gives a contradiction, using Condition \ref{itm:recursion-b}.
    \end{proof}

    \begin{thm} \label{thm:3-3.5}
	    Let $k \ge 5$, $N = \ceil{\left(2k - 3 + \frac{1}{2k-5}\right)n}+1$ and $M = \ceil{\left(2k - 3.5 + \frac{1}{2k-5} \right)n}+1$. Then $K_{N, M} \overset{k}{\longrightarrow} \CM(n+1)$.
    \end{thm}
      
    \begin{proof}
        We use \Cref{lem:recursion}. To that end, let us check that Conditions \ref{itm:N-M-lower} to \ref{itm:recursion-b} hold.
        Condition \ref{itm:N-M-lower}, that $N, M \ge (2k - 3.5)n$, clearly holds. Next, note that
        \begin{align}\label{eqn:specialt-1}
        \begin{split}
            \frac{NM}{n}-(N+M)(k-2)& > (2k-3)\left( 2k-3.5\right)n-\left(4k-6.5\right)n(k-2) \\ 
            & = kn+(k-5)n/2\\
            & \ge kn,
        \end{split}
        \end{align} 
        where we used the fact that $NM/n-(N+M)(k-2)$ is increasing for both $N$ and $M$ whenever $N, M \ge (k-2)n$, and the assumption that $k \ge 5$. This implies that Condition \ref{itm:NM-lower} holds.
        
        Next, notice that 
        \begin{align*}
            N-2n &\ge (2(k-1)-3)n+1 \\
            2M-2n(k-1) - \frac{2n}{2k-5} & \ge (2(k-1) - 3)n + 1,
        \end{align*}
        so Condition \ref{itm:recursion-a} holds by \Cref{thm:2k-3}.
        Similarly,
        \begin{align*}
            2N-2n(k-1) - \frac{2n}{2k-5} & \ge (2(k-1) - 2)n + 1 \\
            M - 2n & \ge (2(k-1) - 3.5)n + 1.
        \end{align*}  
        Thus, Condition \ref{itm:recursion-b} holds by \Cref{thm:3.5-2}, for which we require $k-1 \ge 4$.
        We have seen that all the conditions of \Cref{lem:recursion} hold. The proof follows.
    \end{proof}

    \begin{thm} \label{thm:2k-7/2}
		For $k \ge 6$ we have $r_k(n+1) \le \left(2k  - 3.5 + \frac{1}{k-2}\right)n + 1.$
	\end{thm}
    
    \begin{proof}
        
		We again use \Cref{lem:recursion} with $M = N = \left(2k  - 3.5 + \frac{1}{k-2}\right)n + 1$. Note that Condition \ref{itm:N-M-lower} of \Cref{lem:recursion} holds, as $M, N \ge (2k - 3.5)n$. Next, note that
        \begin{align*}
            \frac{N^2}{n} - 2N(k-2) & > \left( 2k - 3.5 + \frac{1}{k-2} - 2(k-2)\right)N \\ 
            & = \left(\frac{1}{2}+\frac{1}{k-2}\right)N \\
            & > \left(k + \frac{1}{4} + \frac{1.5}{k-2}\right)n \\
            & \ge kn,
        \end{align*}  
        so Condition \ref{itm:NM-lower} holds. Finally, note that
        \begin{align*}
            N - 2n 
            \ge \left(2(k-1) - 3.5 + \frac{1}{k-2}\right)n + 1 
            \ge \left(2(k-1) - 3.5 + \frac{1}{2(k-1) - 5}\right)n + 1,
        \end{align*}
        as $k \ge 5$, and, using $k \ge 6$,
        \begin{align*}
            2N - 2n(k-1) - \frac{2n}{2k-5} 
            & = \left(2k - 5 + \frac{2}{k-2} - \frac{2}{2k-5}\right)n + 2 \\
            & \ge \left(2(k-1) - 3 + \frac{1}{2(k-1)-5}\right)n + 1.
        \end{align*}
        It follows from \Cref{thm:3-3.5} that Conditions \ref{itm:recursion-a} and \ref{itm:recursion-b} hold (note that they are equivalent here, since $N = M$). The proof now follows from \Cref{lem:recursion}.
    \end{proof}
    We note that it is not hard to obtain a slightly weaker result, that still beats \Cref{thm:path-2k-3}, and applies for $k=5$ as well. 
    
\section{From connected matchings to paths and cycles}\label{sec:regularity}
    
    In \cite{luczak99-con-match} {\L}uczak introduced a method that reduces problems about paths and cycles to problems about connected matchings. As this method has become standard, we do not give precise details here, instead we give a brief overview of how to make use of \Cref{thm:2k-3,thm:2k-7/2} to deduce \Cref{thm:path-2k-3,thm:path-2k-3.5} (see, e.g., our previous paper \cite{3-bip-path} on this subject for more details in a very similar setup). 
    
    Let $k$ be an integer, and suppose that we have an upper bound of the form $r_k(n+1) \le \alpha_k n$ which holds for every large enough $n$ (this is what we get from \Cref{thm:2k-3,thm:2k-7/2}). Let $n$ be very large, and consider a $k$-colouring of $K_{N,N}$, denoted $G$, where $N$ is a bit larger than $\alpha_k n$. Our aim is to find a cycle (or path) of length $2n$. 
    
    Apply Szemer\'edi's regularity lemma (see \cite{sze1978-regularity}) to $G$. We obtain a partition of the vertices into a not-too-small and not-too-large number of clusters of almost equal size, such that the graph between almost every pair of clusters behaves almost randomly in every colour. Now consider the graph $\G$, whose vertices represent the clusters, and for which there is an edge of colour $c$ between two clusters if the graph of $c$-coloured edges between them behaves randomly and is somewhat dense. By removing a few clusters and insisting that clusters are subsets of one of the parts of $G$, we may assume that $\G$ is a balanced bipartite graph, whose every vertex is joined to almost every vertex on the other side.
    
    {\L}uczak's interesting observation is that a monochromatic connected matching $\M$ in this graph can be lifted to a path that covers almost all the vertices in the clusters of $\M$ and very few other vertices. Thus, in order to find a path (or, with a little more effort, a cycle) of length $2n$ in $G$, it suffices to find a connected matching of size almost $\frac{1}{2\alpha_k}|\G|$ in $\G$. 
    
    So, our task boils down to proving that $\G$ has a monochromatic connected matching of the required size. The next theorem allows us to reduce the problem of finding monochromatic connected matchings in almost complete bipartite graphs to the same problem in complete bipartite graphs. Thus, by the assumption that $r_k(n+1) \le \alpha_k n$, we are done. \Cref{thm:path-2k-3,thm:path-2k-3.5} follow by the above argument, the following theorem, and \Cref{thm:2k-3,thm:2k-7/2}.

    \begin{thm} \label{thm:cm-almost-complete}
		Let $0< \eps < (8k)^{-2k} = c_k^{-1}$ and $N \ge (1+\eps c_k)r_k(n+1).$ Let $G$ be a subgraph of $K_{N,N}$ of minimum degree at least $N-\eps n.$ Then, in every $k$-colouring of $G$, there is a monochromatic connected $(n+1)$-matching.
	\end{thm}
    
    We omit the proof of this theorem since it is essentially the same  as the one given in our previous paper \cite{3-bip-path}, albeit only for $3$ colours. In \cite{3-bip-path} we also discuss the by now standard connected matching method of {\L}uczak described in this section in more details.

\section{Lower bounds} \label{sec:lower-bounds}

    The authors in \cite{debiasio-multicolour} were interested in the function $f(k)$, defined to be the minimum $N$ such that in every $k$-colouring of $K_{N, N}$ there is a monochromatic $P_4$ (i.e.\ a path of length $3$). Since a bipartite graph is $P_4$-free if and only if it is a star forest, $f(k) - 1$ is the maximum $N$ such that $K_{N, N}$ can be $k$-coloured in such a way that all monochromatic components are stars. They determined $f(k)$ for every $k$, as follows.
    \begin{align} \label{eqn:star-ramsey}
        f(k) = \left\{
            \begin{array}{ll}
                k+1 & k \le 3 \\
                6 & k = 4 \\
                2k - 3 & k \ge 5.
            \end{array}
        \right.
    \end{align}
	The function $f(k)$ is relevant to the study of bipartite Ramsey numbers of paths, cycles, or connected matchings. Indeed, we note the following observation made in \cite{debiasio-multicolour}; here $\rbip{k}{H}$ denotes the $k$-colour bipartite Ramsey number of $H$.
    \begin{obs} \label{obs:lower-bds}
        Let $N = (f(k) - 1)n$. There is a $k$-colouring of $K_{N,N}$ without a monochromatic connected $(n+1)$-matching, which can be obtained by blowing up a construction for $f(k)$, by replacing each vertex by $n$ new vertices. In particular, by \eqref{eqn:star-ramsey}, we have
		\begin{align} \label{eqn:ramsey-lower-bound}
            \rbip{k}{C_{2(n+1)}} \,\ge\, \rbip{k}{P_{2(n+1)}} \,\ge\,  \rbip{k}{\CM(n+1)} \,>\, 
            \left\{
                \begin{array}{ll}
                    kn & k \le 3 \\
                    5n & k = 4 \\
                    (2k-4)n & k \ge 5.
                \end{array}
            \right.
        \end{align}
    \end{obs}
    This establishes the lower bound in \Cref{cor:4-col}. It also shows that the upper bounds in \Cref{thm:path-2k-3,thm:path-2k-3.5} are close to being tight. 

    We believe that the lower bounds in \Cref{obs:lower-bds} could be improved, possibly close to $(2k - 3.5)n$ (thus almost matching the upper bound in \Cref{thm:path-2k-3.5}). For five colours, we indeed obtain such an improvement.
    
    Initially, we found a better example than the one mentioned above for five colours by a computer, using a technique called simulated annealing. The example we present here is constructed by hand, based on observations of the properties of the example found by a computer.
    
    While, by \eqref{eqn:star-ramsey}, it is not possible to $5$-colour $K_{7,7}$ such that all monochromatic components are stars, such a colouring exists for the graph obtained from $K_{7,7}$ by removing a single edge, see \Cref{fig:eg1} below.
    \begin{figure}[ht]
        \caption{A $5$-colouring of the graph $K_{7,7}$ with the edge $u_0v_0$ removed, where all the monochromatic components are stars.}
        \includegraphics[scale=0.8]{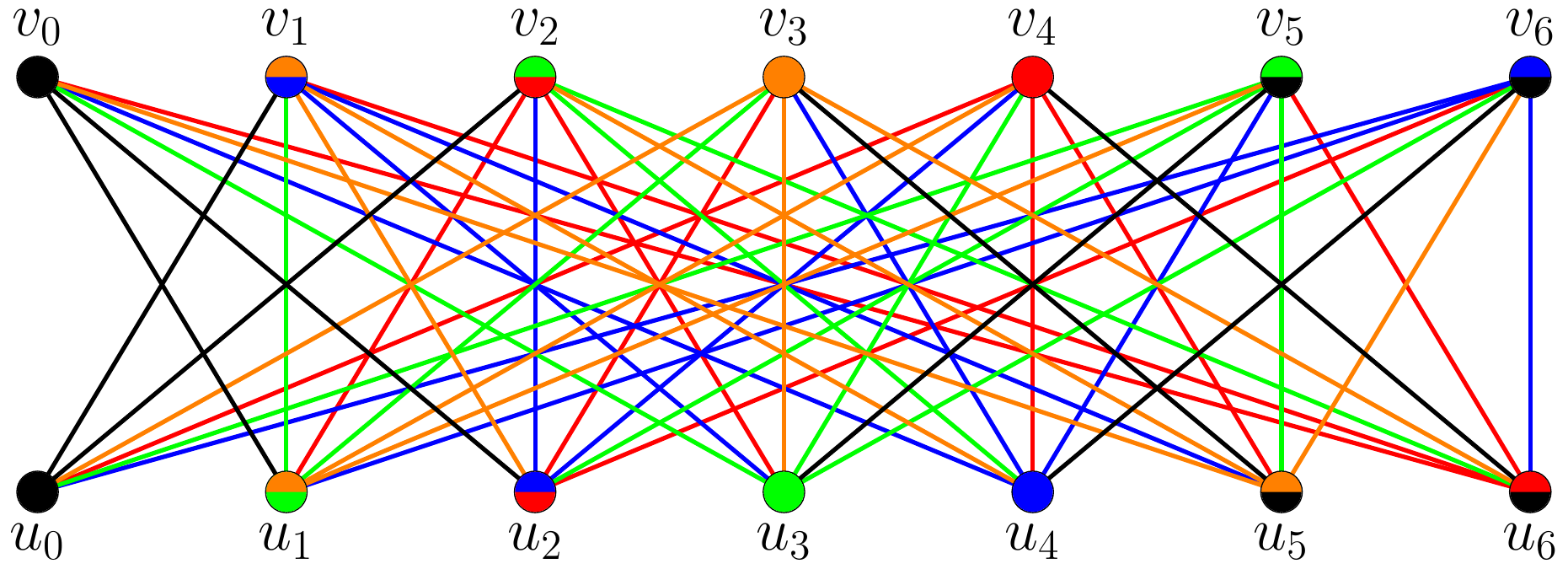}
        \label{fig:eg1}
	\end{figure}
    \begin{figure}[ht]
        \caption{A decomposition of the colouring given in \Cref{fig:eg1} making it more readable (the numbers correspond to the indices of the vertices in \Cref{fig:eg1}; the vertices on the left-hand side are permuted so as to make the colouring in each subfigure look nicer).}
        \includegraphics[scale=.66]{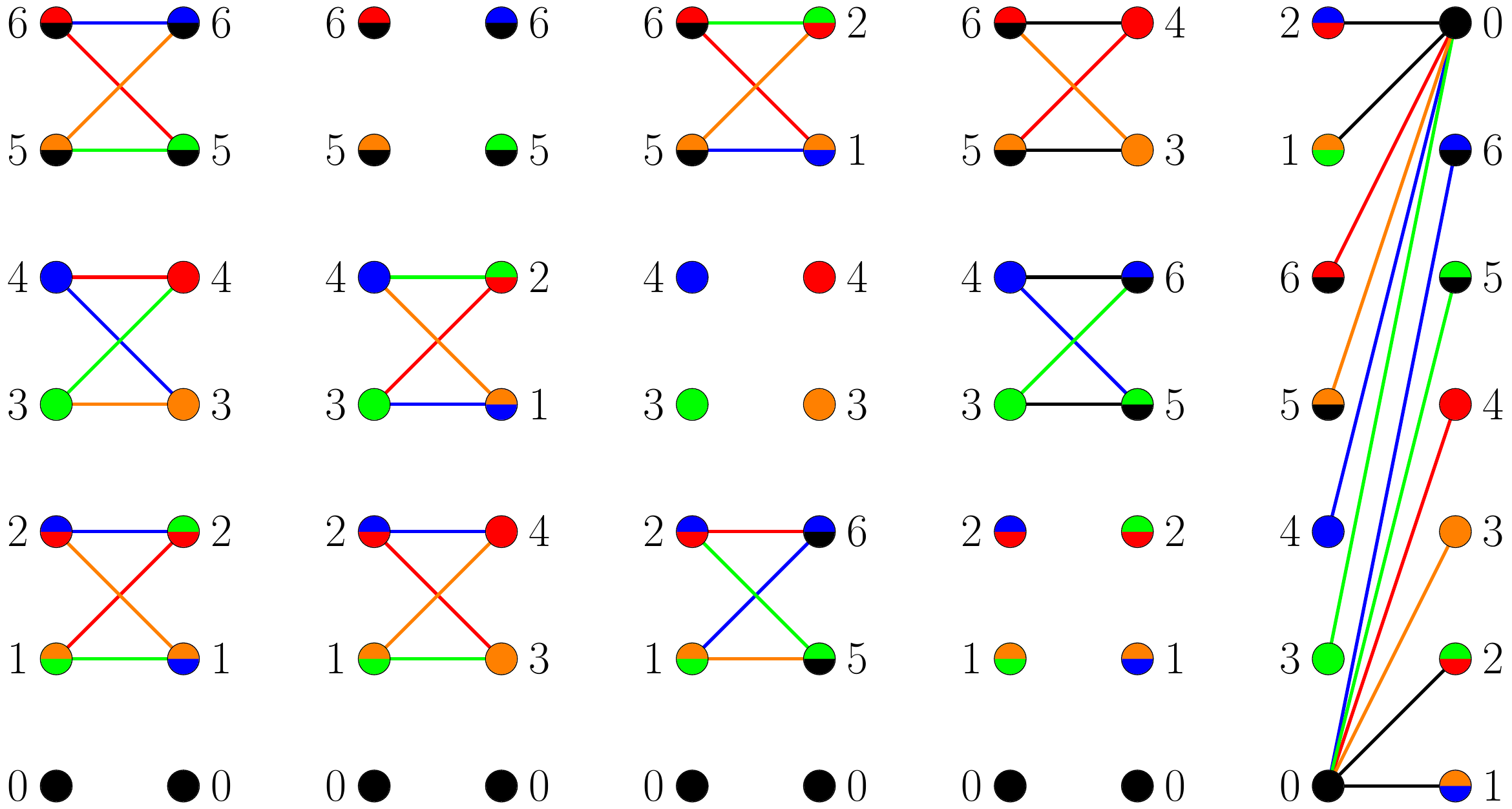}
        \label{fig:eg2}
	\end{figure}
	
	To make it easier to check the claimed properties of the figure, we coloured the vertices as follows: for each monochromatic component, we know it is a star, and we colour a centre of the star by the colour of the component (for stars that consist of a single edge we pick exactly one of its vertices). In order to check that all monochromatic components in this example are stars one only needs to check that for every edge one of its endpoints is coloured in its colour (this might be easier using \Cref{fig:eg2}) and that every vertex has degree exactly one in all colours it is not coloured in (easier in \Cref{fig:eg1}).
    
    We are now ready to prove \Cref{thm:lower-5}, which asserts that the bipartite Ramsey number of a path or a cycle of order $2(n+1)$ is larger than $6.5n$.

    \begin{proof}[ of \Cref{thm:lower-5}]
		Note that the statement holds for $n = 1$, by \eqref{eqn:ramsey-lower-bound}.

        Let $G$ be the graph depicted in \Cref{fig:eg1}, with an additional black edge between $u_0$ and $v_0$. Note that all monochromatic components in $G$ are stars, with the exception of the black component that contains $u_0$ and $v_0$, for which the set $\{v_0, u_0\}$ is a vertex cover.
        
        Replace each vertex $v_i$ (resp.\ $u_i$) with a set of vertices $V_i$ (resp.\ $U_i$) of size $n$ if $i \ge 1$ and of size $\floor{n/2}$ if $i=0$. Next, add all edges between $U_i$ and $V_j$ and colour them by the colour of $u_i v_j$ in $G$. Call the resulting graph $H$.
        
        The graph $H$ is a $5$-coloured $K_{N, N}$, where $N = \floor{6.5n}$, and we claim that $H$ does not contain a monochromatic connected $(n+1)$-matching. Indeed, note that every monochromatic connected component in $H$ is a blow-up of a monochromatic connected component in $G$. Now, the black component that contains $U_0 \cup V_0$ (i.e.\ that corresponds to black component that contains $u_0 v_0$) has a cover $U_0 \cup V_0$, of size $2 \floor{n/2} \le n$. All other components are blow-ups of a star by sets of size at most $n$, in particular they all have covers of size at most $n$. Since, as we have now shown, every connected component has a cover of size at most $n$, it follows that there are no monochromatic connected $(n+1)$-matchings, as needed.
    \end{proof}

\section{Concluding remarks and open problems} \label{sec:conc-remarks}
 
    In this paper we determined, asymptotically, the $4$-colour bipartite Ramsey number of even cycles and consequently for paths. Specifically, we showed here that $\rbip{k}{C_{2n}}=(2k-3+o(1))n$ holds for $k=4$, a bound which also holds for $k = 3$ (as we showed in \cite{3-bip-path}). This is in contrast to the non-bipartite Ramsey numbers, where already the four colours case remains unresolved. We also showed that the behaviour changes as $k$ grows, by showing that $r_k(n+1) \le (2k-3.5+1/(k-2)+o(1))n$ for $k \ge 6$, which is quite close to the lower bound of $(2k-4)n+1$ from \cite{debiasio-multicolour}. We believe that our bound might be asymptotically optimal, which is in part supported by our example for five colours which shows that $r_k(n+1) \ge (2k-3.5)n$ when $k=5$.  Our approach for obtaining upper bounds relies on having at least $kn+1$ special vertices, which guarantees the existence of two components of the same colour that contain special vertices. As the smallest value of $N$ for which this is no longer guaranteed to occur is $(2k - 3.5 + 1/k + O(1/k^2))n$, this value seems to be a natural barrier for our arguments, giving further evidence towards this bound being close to the truth. The most natural next question is to determine what is the correct behaviour for five or more colours. 

    Another interesting direction might be to strengthen our results by obtaining exact bounds for long enough paths or cycles, possibly by obtaining a stability version of our results, similarly to \cite{kohayakawa-simonovits-skokan,ramsey3path,benevides-skokan}.

\providecommand{\bysame}{\leavevmode\hbox to3em{\hrulefill}\thinspace}
\providecommand{\MR}{\relax\ifhmode\unskip\space\fi MR }
\providecommand{\MRhref}[2]{%
  \href{http://www.ams.org/mathscinet-getitem?mr=#1}{#2}
}
\providecommand{\href}[2]{#2}

\end{document}